\documentclass[reqno,11pt]{amsart}
\usepackage{amsfonts}
\usepackage{amsmath}
\usepackage{amsthm}
\usepackage{amssymb}
\usepackage{mathtools}
\usepackage{mathrsfs} 
\usepackage{bbding}
\usepackage{bbm}
\usepackage{dsfont}
\usepackage{upgreek}
\usepackage[latin1,utf8]{inputenc}
\usepackage{libertine}
\DeclareMathAlphabet\EuScript{U}{eus}{m}{n}

\usepackage[numbers]{natbib}

\usepackage{hyperref}
\usepackage[capitalise]{cleveref}
\hypersetup{
    colorlinks=true,
    linkcolor=Violet,
    citecolor=Violet,
    urlcolor=Violet,
}

\DeclareFontFamily{U}{BOONDOX-calo}{\skewchar\font=45 }
\DeclareFontShape{U}{BOONDOX-calo}{m}{n}{
  <-> s*[1.05] BOONDOX-r-calo}{}
\DeclareFontShape{U}{BOONDOX-calo}{b}{n}{
  <-> s*[1.05] BOONDOX-b-calo}{}
\DeclareMathAlphabet{\mathcalboondox}{U}{BOONDOX-calo}{m}{n}
\SetMathAlphabet{\mathcalboondox}{bold}{U}{BOONDOX-calo}{b}{n}
\DeclareMathAlphabet{\mathbcalboondox}{U}{BOONDOX-calo}{b}{n}

\usepackage{fancyhdr}
\usepackage{enumerate}
\usepackage{cancel}
\usepackage{graphicx}
\usepackage{float}
\usepackage[dvipsnames]{xcolor}

\usepackage{accents}

\usepackage{constants}

\newconstantfamily{small}{
symbol=\mathtt{c}
}

\theoremstyle{plain}
\newtheorem{theorem}{Theorem}[section]
\newtheorem{lemma}[theorem]{Lemma}
\newtheorem{proposition}[theorem]{Proposition}
\newtheorem{proposition*}{Proposition}

\newtheorem{claim}{Claim}[section]

\theoremstyle{definition}

\theoremstyle{remark}
\newtheorem{remark}{Remark}[section]

\numberwithin{equation}{section}

\makeatletter
\newcommand{\ltext}[2]{%
  \@bsphack
  \csname phantomsection\endcsname % in case hyperref is used
  \def\@currentlabel{#1}{\label{#2}}%
  \@esphack
}
\makeatother

\DeclareMathOperator{\N}{\mathbb{N}}
\DeclareMathOperator{\Z}{\mathbb{Z}}
\DeclareMathOperator{\R}{\mathbb{R}}

\DeclareMathOperator{\Hh}{\mathcal{H}}
\DeclareMathOperator{\Om}{\Upomega}

\DeclareMathOperator{\A}{\mathscr{A}}
\DeclareMathOperator{\p}{\mathbbm{P}}
\DeclareMathOperator{\E}{\mathbbm{E}}

\title[Competition on a RGG]{Coexistence of Species in a Competition Model on Random Geometric Graphs}

\author[C. F. Coletti \and L. R. de Lima]
{Cristian F. Coletti \and Lucas R. de Lima}

\address{Center for Mathematics, Computation, and Cognition\\
Federal University of ABC\\
Av. dos Estados, 5001\\
09210-580 Santo Andr\'e, S\~ao Paulo\\
Brazil}
\email{cristian.coletti@ufabc.edu.br}

\address{Department of Statistics\\
Institut of Mathematics and Statitics\\
University of S\~ao Paulo\\
Rua do Mat\~ao, 1010\\
05508-090 São Paulo - SP\\
Brazil}
\email{lrdelima@ime.usp.br}
\email{lrdelimath@gmail.com}

\thanks{{\bf Funding:} Research supported by grants \#2017/10555-0, \#2019/19056-2, \#2020/12868-9, \#2023/13453-5, and \#2024/06021-4, S\~ao Paulo Research Foundation (FAPESP).}



\begin{document}

%\nocite{*}

\begin{abstract}
    This paper investigates the coexistence of two competing species on random geometric graphs (RGGs) in continuous time. The species grow by occupying vacant sites according to Richardson's model, while simultaneously competing for occupied sites under the dynamics of the voter model. Coexistence is defined as the event in which both species occupy at least one site simultaneously at any given time. We prove that coexistence occurs with strictly positive annealed probability by applying results from moderate deviations in first-passage percolation and random walk theory, with a focus on specific regions of the space.
\end{abstract}

\maketitle

%\tableofcontents

\section{Introduction}

The study of species coexistence in competition models, particularly those grounded in probabilistic frameworks, provides valuable insights into multispecies dynamics in community ecology models (see \citet{lanchier2024} for a detailed discussion). Probabilistic competition interfaces vary depending on the combination of interacting stochastic processes,  including first-passage competition (introduced by \citet{haggstrom1998}), multitype contact processes (see \citet{neuhauser1992,durrett1997}), and chase-escape models (see \citet{tang2018phase,durret2020}).

The definition of species coexistence in competition models depends on the model’s properties. For instance, in first-passage percolation (FPP) models, coexistence is typically defined as the event in which species occupy infinitely many sites, meaning all species continue to occupy vacant sites as time progresses (as discussed in \citet{garet2005}). The model analyzed in this paper was introduced by \citet{kordzakhia2005} for two species competing on the hypercubic lattice $\Z^d$. It combines a first-passage growth with invasion dynamics governed by the voter model. Here, coexistence is defined as the event where, at any given time, both species occupy at least one site of the graph. This model shares similarities with chase-escape and multitype contact processes but exhibits some particular spatial and temporal properties that influence long-term behavior and species interactions.

In this work, we extend the model to the infinite connected component of random geometric graphs (RGGs), which serve as a random environment for the process, and explore the resulting implications.

A crucial factor used by \citet{kordzakhia2005} was that Richardson's growth model exhibits a uniformly curved limiting shape on $\Z^d$. \citet{coletti2023} established that the asymptotic shape of the same model on RGGs is an Euclidean ball, and \citet{delima2024speed} derived moderate deviations for this setting. We combine these results with additional techniques to control the growth and invasion dynamics, enabling an analysis of the survival of the species within given regions of the space in this random environment.

\subsection{Definition of the competition model and main result}

We consider a class of random geometric graphs in~$\mathbb{R}^d$ defined as follows. Let~$d \in \mathbb{N}$ with~$d \geq 2$,~$\lambda > 0$, and~$r > 0$. The random geometric graph~$\mathcal{G}$ has a vertex set~$V$ generated by a homogeneous Poisson point process on~$\mathbb{R}^d$ with intensity~$\lambda$. The edge set~$\mathcal{E}$ is defined by $\mathcal{E} := \{\{x, y\} : x, y \in V \text{ and } \|x - y\| < r\}$, where $\|\cdot\|$ denotes the Euclidean norm.

As established in the literature (e.g., Penrose~\cite{penrose1996}), there exists a critical threshold $r_c(\lambda) \in (0, +\infty)$ such that, for $r < r_c(\lambda)$, the graph~$\mathcal{G}$ almost surely does not contain any infinite connected component. Conversely, for $r > r_c(\lambda)$,~$\mathcal{G}$ will almost surely exhibit a unique infinite component. In this work, we focus on the supercritical case, $r > r_c(\lambda)$, where we denote this almost surely unique infinite connected component by~$\mathcal{H}$. Given $x\in\R^d$, let $q(x)$ denote the vertex of $\Hh$ that is closest to $x$ in the Euclidean distance. In cases where multiple vertices minimize the distance, an arbitrary but deterministic procedure is used to select $q(x)$.

Let us designate the two competing species as red and blue, each establishes territories within the spatial domain $\Hh$. The occupancy of sites by the red species at any given time $t \in [0, +\infty)$ is symbolized by $\upxi(t)$, while $\upzeta(t)$ represents the analogous territory held by the blue species.

At the outset, $\upchi(t) := \upxi(t) \dot{\cup} \upzeta(t) \subseteq \Hh$ is defined as the combined inhabited territory. The dynamics governing growth and competition are determined by Richardson's and voter's models. Within this framework, the competition unfolds as follows:
\begin{itemize}
    \item Unoccupied sites at time $t$, denoted as $x \not\in \upchi(t)$, are subject to occupation by either species. The rate of occupation by the red or blue species is determined by the presence of neighboring sites already occupied by each species. More specifically, the rates are
    \[\sum_{y \sim x}\mathbbm{1}_{\upxi(t)}(y) \quad \text{and} \quad \sum_{y \sim x}\mathbbm{1}_{\upzeta(t)}(y).\]

    \item Occupied sites at time $t$, specifically $x \in \upchi(t)$, witness potential invasion. Here, the transition to a different species occurs at rates contingent upon neighboring sites' current occupants.
    \begin{itemize}
        \item If $x \in \upxi(t)$, it undergoes a color shift to blue at rate $\sum_{y \sim x}\mathbbm{1}_{\upzeta(t)}(y)$.
        \item Likewise, if $x \in \upzeta(t)$, the site is becomes inhabited by the red species at rate $\sum_{y \sim x}\mathbbm{1}_{\upxi(t)}(y)$.
    \end{itemize}
\end{itemize}

\begin{figure}[htb!]
    \centering
    \includegraphics[width=0.8\linewidth]{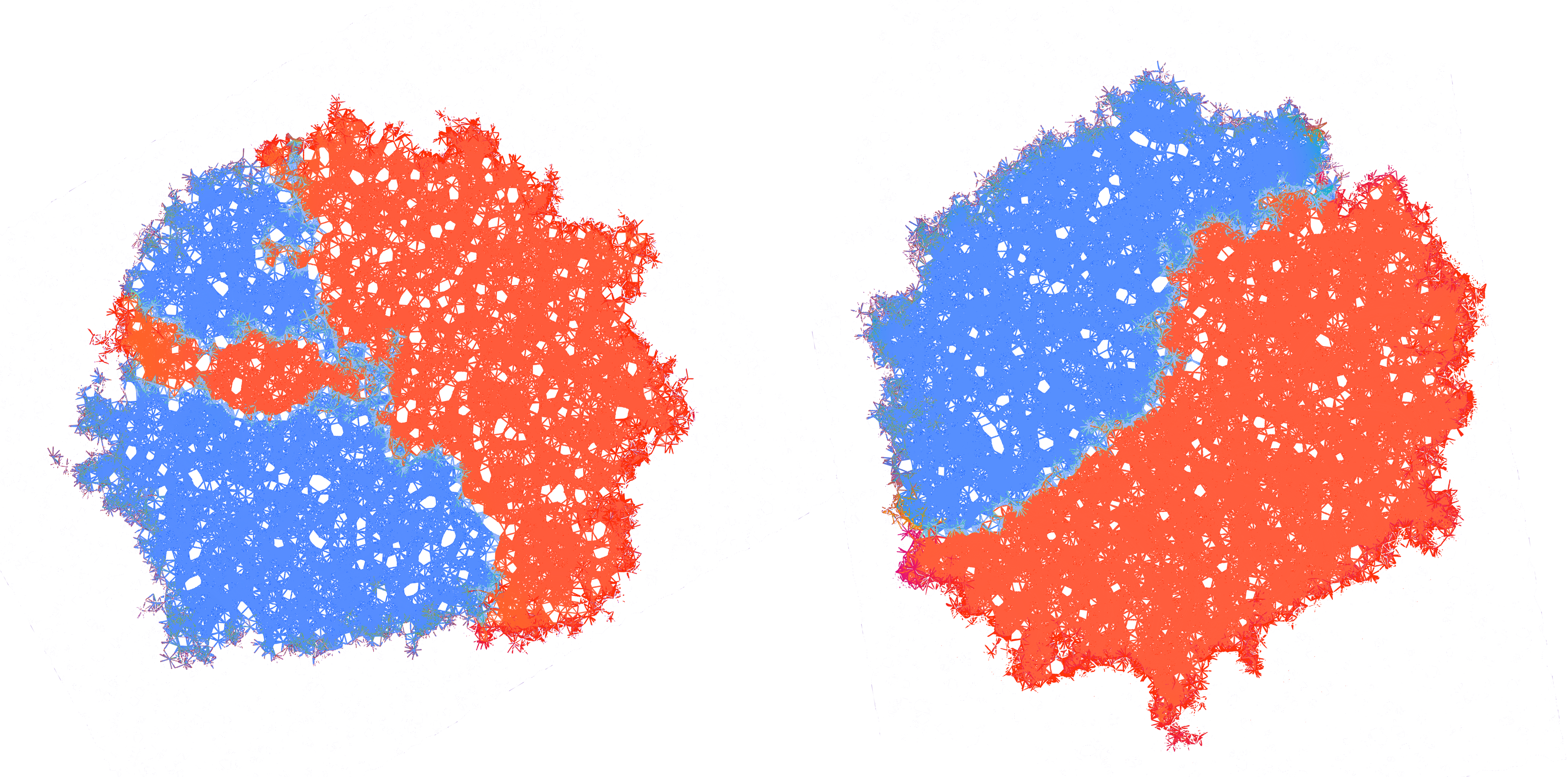}
    \caption{Two simulations of the competition model on a RGG.}
\end{figure}

\subsubsection*{Initial configuration} The initial conditions for $\upxi(0)$ and $\upzeta(0)$ are established by two given sets in the $d$-dimensional space. The two-species competition model is considered to have a \emph{finite initial configuration} if there exist two non-empty disjoint sets $W$ and $W'$ in a finite region of $\R^d$ such that
\[\upxi(0) \subseteq q(W), \ \upzeta(0) \subseteq q(W'), \ \text{and} \ \upchi(0)=q(W \cup W').\]
The configuration of the set $q(W) \cap q(W')$ can be chosen arbitrarily with any random or deterministic rule. A finite initial configuration  for the competition model is called \textit{viable} when $\p(\upxi(0)\neq\varnothing \text{ and }\upzeta(0)\neq\varnothing)>0$.

\begin{remark}
    One can easily verify that finite disjoint sets do not necessarily determine viable finite initial configurations. Consider the deterministic rule that $\upxi(0) = q(W)$ and $\upzeta(0)=q(W')\setminus q(W)$. Then if $W$ is the Euclidean ball $B(o,2r)$ and $W'=\{o\}$, then $\upzeta(0)=\varnothing$ $~\p$-a.s. On the other hand, considering the same sets $W$ and $W'$, if the rule to determine the colouring of $q(W)\cap q(W')$ is random with Bernoulli distribution $p \in (0,1)$, for instance, then the finite initial configuration is viable.
\end{remark}

\begin{remark} \label{rmk:fpp.growth}
    The definition of the initial configuration using the deterministic sets is a straightforward consequence of the FPP model as studied in \cite{coletti2023,delima2024speed}. Once $u\in \R^d$ is fixed, the random variable $T(u,v)$ is the passage time from $q(u)$ to $q(v)$. The growth dynamics of both species given by $\upchi(t)$ is now given by the union of FPP growth originated from $q(W \cup W')$.
\end{remark}

Our primary objective with this model is to determine whether both species coexist with positive probability. To this end, we introduce the event of \textit{coexistence} of both species as
\[\mathrm{Coex}(\upxi,\upzeta):=\left\{\text{for all }t\geq 0, \ \upxi(t) \neq \varnothing \text{ and } \upzeta(t) \neq \varnothing \right\}.\]

We then present the following theorem on the annealed probability of coexistence.
\begin{theorem} \label{thm:coexistence}
    Let $d \geq 2$ and $r>r_c(\lambda)$, and consider the two-species competition model defined above. Then, for any viable finite initial configuration, one has
    \begin{equation*}
        \p\left(\operatorname{Coex}(\upxi,\upzeta)\right)>0.
    \end{equation*}
\end{theorem}

\subsection{Organization of the paper.}

Foundational results on random geometric graphs and first-passage percolation are presented in Section \ref{sec:basic}, which determines the random growth of species in space over time. Section \ref{sec:intermediate.invasion} introduces an intermediate condition that plays a key role in the coexistence analysis and provides results on controlling invasion times. Finally, the proof of the coexistence theorem, which constitutes the main result of the paper, is given in Section \ref{sec:proof}.

\subsection{Notation.}

Throughout the text, we adopt the convention that $\|\cdot\|$ represents the Euclidean norm on $\R^d$. Consider $o$ as the origin of $\R^d$. The competition model on RGGs is assumed to be defined on a probability space $(\Om, \A, \p)$. Here, $\N$ denotes the set of natural numbers $\{1, 2, \dots\}$, and $\N_0$ is defined as $\N\cup\{0\}$. All remaining notation is introduced within the text.

\section{Basic results of FPP on RGGs} \label{sec:basic}

Consider the random set $\mathcal{P}_\lambda$, representing points in $\mathbb{R}^d$ generated by a homogeneous Poisson point process (PPP) with intensity $\lambda > 0$. Let us revisit the definitions provided in the introduction. by defining the random geometric graph (RGG)~$\mathcal{G} = (V, \mathcal{E})$ on $\mathbb{R}^d$ as follows:
\[
  V = \mathcal{P}_\lambda \quad \text{and} \quad
  \mathcal{E} = \big\{\{u,v\} \subseteq V: \|u-v\|<r,~ u \neq v\big\}.
\]
Since $\lambda^{-\frac{1}{d}}\mathcal{P}_\lambda \sim \mathcal{P}_1$, we consider $\lambda$ as a fixed parameter and allow $r$ to vary due to the homogeneity of the norm. Thus, unless specified otherwise, we set $\lambda=1$ and denote $\mathcal{P}_1$ as $\mathcal{P}$.% Let $\left(\Upxi,\F,\mu\right)$ denote the probability space induced by the construction of $\mathcal{P}$. 

Our objective is to study the spread of infection within an infinite connected component of $\mathcal{G}$. According to continuum percolation theory (see Penrose \cite[Chapter~10]{penrose1996}), for all $d \geq 2$, there exists a critical $r_c(\lambda)>0$ (or $r_c$ for $\lambda=1$) such that $\mathcal{G}$ has an (unique) infinite component $\Hh$ $\p$-\textit{a.s.} for all $r>r_c$. Denoting the sets of vertices and edges of $\Hh$ as $V(\Hh)$ and $\mathcal{E}(\Hh)$ respectively, we often use $\Hh$ to represent $V(\Hh)$.

Let $\mathbbm{B}(t)$ denote the hypercube $[-t/2, t/2]^d$, and consider the Euclidean ball denoted as $B(x,t) := \{y \in \R^d \colon \|y-x\| < t\}$. Fix $\theta_{r}$ as $\p \big( B(o,r) \cap \Hh \neq \varnothing \big)$.  The following proposition presents a fundamental result concerning the volume of $\Hh$, which is a weaker version of Theorem 1 in Penrose and Pisztora~\cite{penrose1996}.

\begin{proposition}\label{prop:Hn.growth}
    Let $d \geq 2$, $r>r_c$ and $\varepsilon \in (0,1/2)$. Then, there exists $\mathtt{c}_0>0$ and $t_0>0$ such that, for all $t \geq s_0$,
    \[
        \p\left((1-\varepsilon)\theta_{r} <\frac{|\Hh \cap \mathbbm{B}(t) |}{t^d}  <(1+\varepsilon)\theta_{r} \right) \geq 1 -\exp(-\mathtt{c}_0t^{d-1}).
    \]
\end{proposition}

\subsection{Control of the growth} 

The occupation of new sites in the competition model is characterized by first-passage percolation, as highlighted in Remark \ref{rmk:fpp.growth}. Consider a non-negative random variable $\tau$. On a realization of the random graph $\mathcal{G} = (V, \mathcal{E})$, assign independent random variables $\{\tau_e : e \in \mathcal{E}\}$ to the edges, all sharing the distribution of $\tau$. The first-passage time between two vertices \(u\) and \(v\) in $\Hh$ is defined as  
$T(u, v) = \inf_{\gamma} \sum_{e \in \gamma} \tau_e$, 
where the infimum is taken over all paths $\gamma$ connecting $u$ to $v$.  

To extend $T(\cdot, \cdot)$ to points in $\R^d$, assign each point $x \in \R^d$ to its nearest vertex $q(x)$ in $\Hh$ based on Euclidean distance. If multiple vertices are equidistant from $x$, resolve the ambiguity using a fixed, deterministic rule. For any $x, y \in \R^d$, define  
\[
T(x, y) = T(q(x), q(y)).
\]  
This extension produces a random metric on $\R^d$ and the random balls
\begin{equation*}
H_t(x):=\{y \in \R^d: T(x,y) \le t\},\quad t \ge 0,
\end{equation*}

As a shorthand, write $T(x)$ for $T(o, x)$ and $H_t$ for $H_t(o)$, where $o$ is the origin. Consider the following conditions on the passage time $\tau$:
\begin{itemize}
    \item[(${A}_1$)] \ltext{${A}_1$}{A1} $\p(\tau=0) =0$;
    \item[(${A}_2$)] \ltext{${A}_2$}{A2}
    $\E[e^{\eta \tau}] < +\infty$ for some $\eta>0$.
\end{itemize}

The assumptions stated earlier are instrumental in establishing the quantitative shape theorem for FPP on RGGs, as presented in Theorem 1.1 of \citet{delima2024speed}. For completeness, we restate it below:

\begin{theorem} \label{thm:speed.FPP}
	Let $d \geq 2$,~$\lambda > 0$ and $r>r_c(\lambda)$, and consider first-passage percolation on the random geometric graph on~$\R^d$ with parameters~$\lambda$ and~$r$, with passage times satisfying~\eqref{A1} and~\eqref{A2} above. 
Then, there exists $c'>0$ and $\upvarphi>0$ such that almost surely, for~$t$ large enough we have
    \begin{equation*}% \label{eq:asymptotic.cone2}
	    \left(1-c'\frac{\log(t)}{\sqrt{t}}{}\right) B(o, \upvarphi) \subseteq \frac{1}{t}H_{t} \subseteq \left(1+c'\frac{\log(t)}{\sqrt{t}}{}\right) B(o, \upvarphi).
    \end{equation*}
\end{theorem}

Next, we combine Theorems 1.2 and 1.3 from \citet{delima2024speed} to derive a result on the moderate deviations of first-passage times:

\begin{theorem}[Moderate deviations of first-passage times] \label{thm_new_moderate_deviations}
	Consider first-passage percolation as in Theorem~\ref{thm:speed.FPP}, under the same assumptions as in that theorem. There exist~$C,C', c > 0$ such that for any~$x \in \mathbb R^d$ with~$\|x\|$ large enough, we have
	\begin{equation*}
		\p\left(\frac{|T(x) - \|x\|/\varphi|}{\sqrt{\|x\|}} > \ell \right) \le Ce^{-c \ell} \quad \text{for any } ~\ell \in \left[C'\log\big(\|x\|\big),\sqrt{\|x\|}\right].
	\end{equation*}
\end{theorem}

Since the growth in the competition model can be described by  $\tau\sim \operatorname{Exp}(1)$, which satisfy \eqref{A1} and \eqref{A2}, the theorems above also apply to $T(u,v)$ and $H_t(u)$ for $u\in W\cup W'$. Furthermore, these results can be extended to the entire region defining $\upchi(t)$, which represents the entire growth of the competition model. This extension is supported by the polynomial volume associated with $W \cup W'$ (see \cref{prop:Hn.growth}). Specifically, $\upchi(t)$ is given by $q\left(\bigcup_{u\in W\cup W'}H_t(u)\right)$.

\section{Intermediate results of the competition model}\label{sec:intermediate.invasion}

In this section, we delve into the competition model presented in the introduction. Our focus is on exploring the dynamics and conditions under which two species can coexist within the framework of RGGs. Throughout the lemmas in this section, we assume the conditions of \cref{thm:coexistence}.

\subsection{Intermediate condition}

To investigate the coexistence of two species within this competition model over time, we will demonstrate that for a given sequence $t_n$ of time instances, it is possible to observe both species inhabiting specific regions of the space. This phenomenon is ensured by what we refer to as the \emph{intermediate condition}.

Let $\mathfrak{a}, \mathfrak{b} \in \left(\frac{3}{4}, ~1\right)$ be such that $\mathfrak{b}<2\mathfrak{a}-1$, and fix $\mathfrak{d} \in (0,1)$. Consider $\{t_n\}_{n\in\N_0}$ and $\{r_n\}_{n\in\N_0}$ to be sequences of times and angles, respectively, taking values in $(0,+\infty)$ such that, for all $n \in \N_0$, \[t_{n+1}-t_n = \mathfrak{d} t_n \quad \text{and} \quad r_n-r_{n+1}= (t_n)^{\mathfrak{a}-1}\] with $t_0>0$ to be determined by \eqref{intermediate.condition}. Then,  $t_n$ is given by $t_n :=t_0(1+\mathfrak{d})^n$ and we set \[r_n:=\frac{1+(1+\mathfrak{d})^{(\mathfrak{a}-1)\cdot n}}{1-(1+\mathfrak{d})^{(\mathfrak{a}-1)}}(t_0)^{\mathfrak{a}-1}.\]
Observe that $r_n$ decreases to $r_0/2>0$ as $n \uparrow +\infty$.

Fix $\bar{s} := \left(\frac{2(1+\mathfrak{d})}{\mathfrak{d}}\right)^{\frac{1}{1-\mathfrak{a}}}$ and let $\mathcal{A}(o,t,t')$ be the annulus $B(o,t')\setminus B(o,t)$. For $z \in \partial B(o,1)$ and $t_0 > \bar{s}$, define $\Upphi_n(z)$ as the random set of vertices in a region given by 
\[
    \Upphi_n(z) := \upvarphi\cdot\mathcal{A}\left(o, ~\tfrac{1}{1+\mathfrak{d}}t_n,~t_n-(t_n)^{\mathfrak{b}}\right) \cap \operatorname{Cone}(z, r_n)\cap \Hh.
\]
Let $\Uptheta_{w,w'}$ be the  event that guarantees the existence of vertices in regions of interest, defined as follows:
\[\Uptheta_{w,w'}:=\bigcap_{n\in\N_0} \big\{\Upphi_n(w)\neq\varnothing \ \text{ and } \ \Upphi_n(w')\neq \varnothing\big\}.\]
Consider $\widehat{H}$ to be the set $q^{-1}(H)$ for $H \subseteq \Hh$. Define $\Gamma_n$ as the event associated with convergence to the limiting shape, adjusted for $\upchi(0)$:
\[\Gamma_n :=\left\{B\left(o, \upvarphi(t_n - (t_n)^{\mathfrak{b}})\right) \subseteq \widehat{\upchi}_n \subseteq B\left(o, \upvarphi(t_n + (t_n)^{\mathfrak{b}})\right)\right\}.\]

Before stating the intermediate condition, we will restrict our attention to a specific part of the occupied sites $\upchi$. Note that, since $t_0>\bar{s}$, it follows that  $r_n \in(0, \pi/2)$ for $n\in \N_0$. Let us write, for all $n \in \N_0$, 
\[\upchi_{z,n}:=\upchi(t_n)\cap \left(\mathrm{Cone}(z,r_n)\left\backslash B\left(o,\tfrac{\upvarphi}{1+\mathfrak{d}}t_n\right)\right.\right).\] 

\begin{figure}[htb!]
    \centering
    \includegraphics[trim={70pt 70pt 70pt 70pt},clip,width=0.72\linewidth]{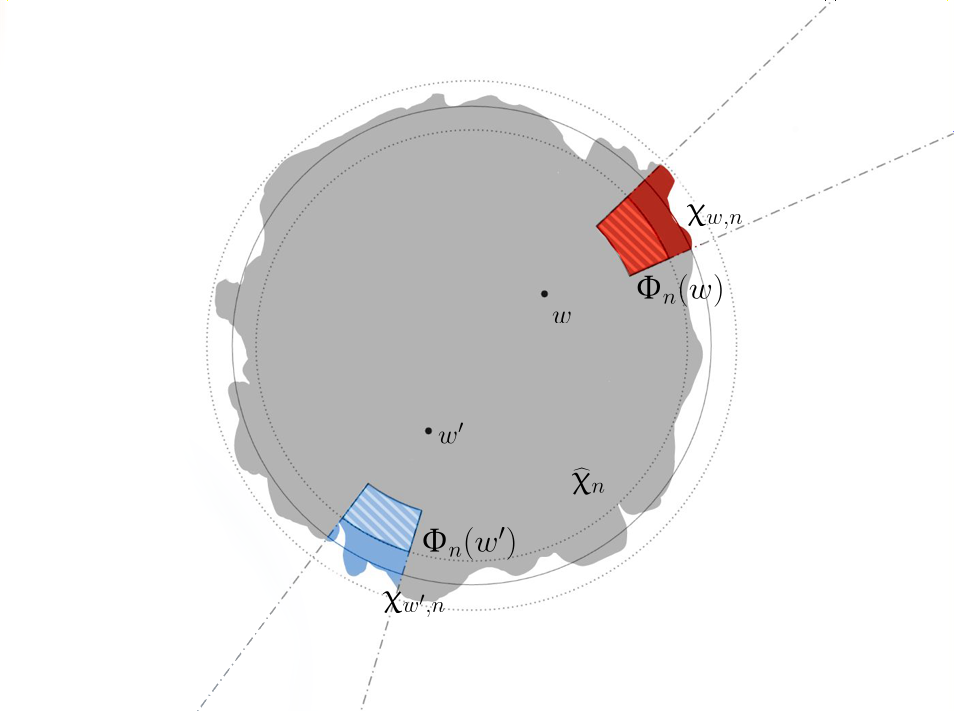}
    \caption{The regions $\upchi_{w,n}$, $\upchi_{w',n}$, $\Upphi_n(w)$, $\Upphi_n(w')$ on the event $\Gamma_n$.}
    \label{fig:comp_intermediate_condition}
\end{figure}

Now, we define $\Uppsi_{w,w'}$ by combining the events above, focusing on the occupation of selected regions of space by a unique species in each region. Specifically,
\[
    \Uppsi_{w,w'}:=\left\{
   \upchi_{w,0}\subseteq\upxi(t_0) \ \text{ and } \ \upchi_{w',0}\subseteq\upzeta(t_0)\right\} \cap \Uptheta_{w,w'} \cap \Gamma_0.
\]

Let the constant $s_0>\bar{s}$ be determined later (see \cref{lm:intermediate.condition} and the Proof of \cref{thm:coexistence} for details). Recall that $\uptheta(u,v)$ denotes the angle between $\vec{ou}$ and $\vec{ov}$. We can now state the intermediate condition for the competition model as follows: There exist $w,w' \in \partial B(o,1)$ and $t_0 \geq s_0$ such that $\uptheta(w,w')>2r_0$ and 
\begin{equation} \tag{\small \scshape I$_0$} \label{intermediate.condition}
    \p\big(\Uppsi_{w,w'}\big)>0.
\end{equation}

Furthermore, observe that $\Upphi_n(w) \subseteq\upchi_{w,n}$ and $\Upphi_n(w') \subseteq\upchi_{w',n}$ (see \cref{fig:comp_intermediate_condition}).

To ensure the occupancy of the regions highlighted above by the red and blue species, we define a \textit{growth-invasion path}. Consider ${\mathbb{G}=(\mathbb{V},\mathscr{E})}$ to be a connected graph, and let $\sigma \in \{0,1,2\}^{\mathbb{V}}$ be a configuration such that, for all $x \in \mathbb{V}$,
\[\sigma(x)=\left\{\begin{array}{cc}
     2, & x\text{ is blue};\\
     1, & x\text{ is red};  \\
     0, & x\text{ is unoccupied.}
\end{array}\right.\]

A \textit{finite growth-invasion path} is a sequence $\{\sigma\}_{i=0}^k$ of configurations in $\{0,1,2\}^{\mathbb{V}}$ where, for all $i \in \{1, \dots,k\}$, there exist a unique $x\in \mathbb{V}$ and an edge $\{x,y\}\in\mathscr{E}$ such that:
\begin{itemize}
    \item[(i)] $\sigma_i(x)-\sigma_{i-1}(x)\neq 0$, and
    \item[(ii)] $\sigma_i(x)=\sigma_{i-1}(y) \in\{1,2\}$.
\end{itemize}
In other words, the red and blue species grow by occupying unoccupied neighbouring vertices, and an invasion occurs when the neighbouring vertex is already occupied.

We now present a lemma that establishes the existence of a finite growth-infection path leading to a certain configuration of interest.

\begin{lemma} \label{lm:irreducibility}
    Let ${\mathbb{G}=(\mathbb{V},\mathscr{E})}$ be a finite connected graph , and suppose there exists $x\in\mathbb{V}$ with degree $\operatorname{deg}(x) \geq 3$. Then, for any configuration $\sigma_0 \in \{0,1,2\}^{\mathbb{V}}$ such that $\sigma_0^{-1}(\{1\})$ and $\sigma_0^{-1}(\{2\})$ are non-empty, there exists a finite growth-invasion path $\{\sigma_i\}_{i=0}^k$ such that: \[\sigma_k\in\{1,2\}^{\mathbb{V}} \quad \text{ and } \quad \sigma_k(x)\in\big\{\sigma_k(y)\colon y \text{ is a neighbour of } x\big\}.\]
\end{lemma}
\begin{proof}
    To prove the existence of a finite growth-invasion path, which may not necessarily be optimal, we analyse a subgraph structure of $\mathbb{G}$. Consider the graph $\mathbb{G}$ with the vertex $x$ removed. The resulting graph is decomposed into at most $\operatorname{deg}(x)$ connected components. Next, fix a spanning tree for each connected component, and let $\{x_j\}_{j=1}^{\operatorname{deg}(x)}$ denote the set of neighbours of $x$ in $\mathbb{G}$.

    If more than one neighbour of $x$ belongs to the same spanning tree, we cut the branch at one of these neighbours, ensuring that we obtain exactly $\operatorname{deg}(x)$ distinct trees. Denote these trees as $\mathbb{T}_j$, where each $\mathbb{T}_j$ is rooted at a different neighbour $x_j$ of $x$.

    Now, let $\mathbb{T}$ be the graph formed by the union of the trees $\mathbb{T}_1, \dots, \mathbb{T}_{\operatorname{deg}(x)}$ together with the star graph with centre vertex $x$ and extremities $\{x_j\}_{j=1}^{\operatorname{deg}(x)}$. Hence, $\mathbb{T}$ is a spanning tree of $\mathbb{G}$ that preserves the degree of $x$. 

    Consider the configurations $\sigma_0$ and $\sigma_k$ fixed as stated above. Suppose $\sigma_0 \in \{0,1,2\}^{\mathbb{V}}$ with non-empty $\sigma_0^{-1}(\{1\})$ and $\sigma_0^{-1}(\{2\})$. Observe that one can easily construct a finite growth-invasion path $\{\sigma_i\}_{i=0}^{k'}$ such that $\sigma_{k'}(x_1) = 1$ and $\sigma_{k'}(x_2) = 2$.

    \begin{figure}[htb!]
        \centering
        \includegraphics[width=0.35\linewidth]{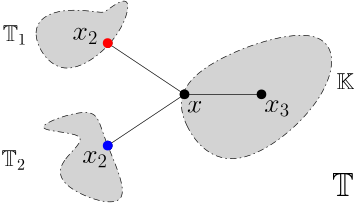}
        \caption{The spanning tree $\mathbb{T}$ of $\mathbb{G}$ and its subgraphs.}
        \label{fig:comp_T}
    \end{figure}

    Let $\mathbb{K}$ denote the subtree of $\mathbb{T}$ rooted at $x$ obtained by removing the vertices of $\mathbb{T}_1$ and $\mathbb{T}_2$. We construct a finite growth-invasion path $\{\sigma_i\}_{i=k'}^{k''}$ with the property that for any vertex $y$ in $\mathbb{K}$, $\sigma_{k''}(y) = \sigma_k(y)$. This is done by alternating growth-invasion paths between $x_1$ and $x_2$ through $x$, reaching all the leaves and branches of $\mathbb{K}$.

    If necessary, we repeat the procedure for the trees $\mathbb{T}_1$ and $\mathbb{T}_2$ by assigning the colours red and blue to $x_3$, respectively. At this stage, we obtain a growth-invasion path $\{\sigma_i\}_{i=0}^{\overline{k}}$ with the property that $\sigma_{\overline{k}}(y) = \sigma_k(y)$ for all $y \in \mathbb{V} \setminus \{x, x_1, x_3\}$. Finally, we alternate the colours of $x_1$, $x_2$, and $x_3$ in a growth-invasion path $\{\sigma_i\}_{i=\overline{k}}^{k}$ within the claw graph formed by $x$, $ x_1$, $x_2$, and $x_3$, where the last step may involve the final change of colour of $x$, adopting the colour of one of its neighbours.

\end{proof}

We will show that we can control the coexistence of both species within suitable regions of space given by the intermediate condition \eqref{intermediate.condition}. The following lemma ensures that the intermediate condition is reasonable for any viable finite initial configuration of the competition model.

\begin{lemma} \label{lm:intermediate.condition}
    Let $W,W'$ be two non-empty disjoint sets of $~\R^d$ determining a viable finite initial configuration of the competition model. Then, for all $\varepsilon\in(0,1)$ and any fixed $w,w'\in \partial B(o,1)$, there exists $s_0 > \bar{s}$ such that, for all $t_0 \geq s_0$,
    \[\p\left(\Uptheta_{w,w'} \cap \Gamma_0\right)>1-\varepsilon \quad \text{and} \quad \p\big(\Uppsi_{w,w'}\big) >0.\]
\end{lemma}

\begin{proof}
    First, let us fix, without loss of generality, an arbitrary $w \in\partial B(o,1)$ and let $w'=-w$. Define $\Gamma$ to be the event $\bigcap_{n\in\N}\Gamma_n$.
    Since $W$ and $W'$ are in a finite region of $\R^d$, \cref{thm:speed.FPP,thm_new_moderate_deviations,prop:Hn.growth} ensure that we can approximate $\p(\Gamma \cap \Uptheta_{w,w'})$ to 1 as closely as we want as $s_0 \uparrow+\infty$. Let $s_0$ be large so that $\p(\Gamma^c \cup \Uptheta_{w,w'}^c) < \varepsilon$.

    It remains to verify that $\upchi_{w,0}\subseteq\upxi(t_0)$ and $\upchi_{w',0}\subseteq\upzeta(t_0)$ occur with positive probability. Define
    \[
    \Upphi_n^+(z) := \upvarphi\cdot\mathcal{A}\left(o, ~\tfrac{1}{1+\mathfrak{d}}t_n,~t_n+(t_n)^{\mathfrak{b}}\right) \cap \operatorname{Cone}(z, r_n)\cap \Hh.
\]

    Note that $\Upphi_n(z) \subseteq \upchi_{z,n} \subseteq \Upphi_n^+(z)$ on $\Gamma_n$. Consider $\upxi(0)$ and $\upzeta(0)$ to be non-empty. Let $\mathcal{W}_{\upxi,w}$ to be a connected subgraph of $\Hh$ that connects $\upxi(0)$ to $\Upphi_0^+(w)$ with $T$-geodesic paths. Similarly, define $\mathcal{W}_{\upzeta,w'}$ to be a subgraph connecting $\upzeta(0)$ to $\Upphi_0^+(w')$.

    If $\mathcal{W}_{\upxi,w}$ and $\mathcal{W}_{\upzeta,w'}$ do not intersect, then $\upchi_{w,0}\subseteq\upxi(t_0)$ and $\upchi_{w',0}\subseteq\upzeta(t_0)$ happens with positive probability by forbidding the invasion dynamics from occurring  with the elements of $\mathcal{W}_{\upxi,w}$ and $\mathcal{W}_{\upzeta,w'}$ before time $t_0$. Let $\mathcal{W}$ be the graph union of $\mathcal{W}_{\upxi,w}$ and $\mathcal{W}_{\upzeta,w'}$. Consider now that case where $\mathcal{W}$ is connected.

    The worst case scenario is when $\mathcal{W}$ is a line graph that can prevent a growth-invasion path from $\upxi(0)$ and $\upzeta(0)$ to $\upchi_{w,0}$ and $\upchi_{w',0}$, respectively (see \cref{fig:comp_line}). The referred paths can be obtained using Harris' graphical construction along with the growth-invasion dynamics where each step (or arrow) happens with distribution $\operatorname{Exp}(1)$.

    \begin{figure}[htb!]
        \centering
        \includegraphics[width=0.5\linewidth]{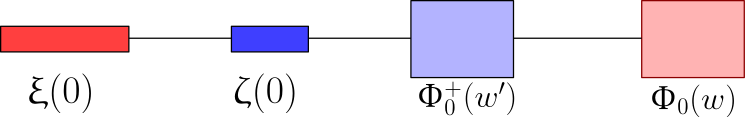}
        \caption{Schematics for worst case of the growth-invasion dynamics in $\mathcal{W}$.}
        \label{fig:comp_line}
    \end{figure}

    To verify the claim above, consider that there exists $x \in \mathcal{W}$ with $\operatorname{deg}(x)\geq 3$. Let us fix the random set
    \[
        \Upphi_0^-(z) := \upvarphi\cdot\mathcal{A}\left(o, ~\tfrac{1}{1+\mathfrak{d}}t_n + r/\varphi,~t_n-(t_n)^{\mathfrak{b}}-r/\varphi\right) \cap \operatorname{Cone}(z, r_n/2)\cap \Hh.
    \]
    
    Consider $\Upsigma_{w,w'}$ to be the event that every site in $\Upphi^-_0(w) \cup \Upphi^-_0(w')$ has degree at most $2$. Due to the properties of the homogeneous Poisson point process, one can choose $s_0> \bar{s}$ so that, for $t_0 \geq s_0$, the probability of $\Upsigma_{w,w'}$ occurring can be made arbitrarily small. Let $E_{w,w'}:= \Uptheta_{w,w'}\cap\Gamma\cap\Upsigma_{w,w'}^c$. Then, the subgraph $\mathcal{W}$ on $E_{w,w'}$ contains at least one vertex $x \in \upchi_{w,0}$ and one vertex $y \in \upchi_{w',0}$, each with degree greater than 2, such that all neighbors of $x$ are in $\upchi_{w,0}$ and all neighbors of $y$ are in $\upchi_{w',0}$.
    
    Observe that \cref{lm:irreducibility} guarantees the existence of a finite growth-invasion path $\{\sigma\}_{i=0}^k$ in $\mathcal{W}$ conditioned on $E_{w,w'}$. Note that the length $k$ of the path is random, depending on the realization of $\mathcal{W}$. By applying \cref{prop:Hn.growth}, one can identify a suitable event that bounds the size of $\mathcal{W}$, and consequently the length $k$.
    
    Furthermore, since each step of the finite growth-invasion path $\{\sigma_i\}_{i=0}^k$ corresponds to an independent exponential random variable with rate 1, occurring within the time interval $(0,t_0]$, and because species outside of $\mathcal{W}$ are forbidden from invading or occupying sites within $\mathcal{W}$, we obtain the following probability bound:
    \begin{equation*}
        \p(\Uppsi_{w,w'} \cap E_{w,w'}\mid\upxi(0)\neq\varnothing \text{ and }\upzeta(0)\neq \varnothing)>0.
    \end{equation*}
    This leads to the desired conclusion.
\end{proof}

\subsection{Control of invasion times}

One of the key aspects of the competition model is controlling the invasion dynamics within the graph. We will utilize its correspondence to random walks to derive the necessary bounds that enable effective control of the invasion.

Consider $S^u_t$ to be a continuous-time simple random walk on a graph ${\mathbb{G}=(\mathbb{V},\mathscr{E})}$ with . Let $\deg_{\max}(\mathbb{G})$ be the maximum vertex degree of $\mathbb{G}$ and denote by $D_{\mathbb{G}}(u, v)$ the graph distance between $u, v \in \mathbb{V}$. The following lemma is a straightforward consequence of Corollary 11 of \citet{davies1993} by using heat kernel techniques (see \cite{barlow2017} for details).

\begin{lemma} \label{lm:heat.kernel} 
    Let $\mathbb{G}$ be a graph with bounded vertex degree, then
        \[\p(S^u_t=v) \leq \operatorname{deg}_{\max}(\mathbb{G})\cdot\exp\left(-\frac{D_{\mathbb{G}}(u,v)^2}{2t} \left(1-\frac{D_{\mathbb{G}}(u,v)^2}{10t^2}\right)\right).\]
\end{lemma}

We apply this result in the proof of the lemma below. First, we define the random ball $\bar{B}_D(x, s):=\{y\in\Hh \colon D(x,y) \leq s\}$ where $D(u,v)$ is the random graph distance between $q(u)$ and $q(v)$. The following result bounds the probability that a region occupied by one species will be invaded within a given time, in relation to its volume.

\begin{lemma} \label{lm:control.invasion}
    Let $x \in \R^d$ and $\rho > 0$ with $B(x,\rho^{\mathfrak{b}}) \cap \Hh \subseteq \upxi(0)~$. Then there exist constants $C,C'>0$ such that, for all $t \in [0,\rho]$,
    \begin{equation} \label{eq:invasion.q(x)}
        \p\big(q(x) \in \upzeta (t)\big) \leq C \exp\left( -C' ~\rho^{2{\mathfrak{b}}-3/2} \right).
    \end{equation}
\end{lemma}
\begin{proof}
     It suffices to consider by stochastic domination that $\upxi(0) = B(x,\rho^{\mathfrak{b}}) \cap \Hh$ and $\upzeta(0) = \Hh\setminus \upxi(0)$. Recall that the dual process of the voter model is the coalescent random walk. Therefore, one can prove the lemma verifying properties of random walks on the infinite component. Let us denote by $S_t^{x'}$ a continuous-time simple random walk on $\Hh$ starting at $q(x')$. Set 
     \[\uptau^x_s := \inf \{t'>0 : S_{t'}^x \not \in B(x,s^{\mathfrak{b}})\}.\]
     We will verify \eqref{eq:invasion.q(x)} applying the following claim.
     
    \begin{claim} \label{cl:first.exit}
         There exist $c_1',c_2'>0$ such that
        \begin{equation} \label{eq:first.exit.p.bound}
            \p\left(\uptau^x_\rho < t\right) \leq c_1' \exp\left( -c_2' ~\rho^{2{\mathfrak{b}}-1} \right).
        \end{equation}
    \end{claim}
    \begin{proof} \renewcommand\qedsymbol{$\blacksquare$}
    Consider w.l.o.g. that $\rho \geq 1$, we will treat the case $\rho <1$ separately. In order to verify \eqref{eq:first.exit.p.bound}, we will first study $S_t^x$ on events with suitable properties for $\Hh$. The strategy is to control the distribution of points and vertex degrees within a given region of the graph.
    
    By Theorem 2.2 of \cite{yao2011} and Palm calculus, there exist $c'_0>1$ and $c''_1,c''_2>0$ such that, for all $s\geq1$,
    \begin{equation} \label{eq:D.Euclid.Lm.Yao}
        \p\big(B(x,s)\cap\Hh \not\subseteq \bar{B}_D(x, c'_0s)\big) \leq c''_1\exp(-c''_2 s).
    \end{equation}
    Set
    \[
        E_1 := \left\{\|q(x)-x\|<\rho^{\mathfrak{b}}/2,~ B(x,\rho^{\mathfrak{b}})\cap\Hh \subseteq \bar{B}_D(x, c'_0~\rho^{\mathfrak{b}}) \right\}.
    \]
    Let $c'':= \frac{2e+1}{2}c_0'r$ and set $\widetilde{\mathbbm{B}}(x,\rho) := x+c''\rho^2[-1,1]^d$ and denote by $\deg(u)$ the degree of the vertex $u \in \Hh$. Write $\Hh_{x,\rho}$ for the subgraph of $\Hh$ restricted to the vertices $\widetilde{\mathbbm{B}}(x,\rho) \cap \Hh$. %and let $\deg_{\max}(\mathcal{G})$ be the maximum vertex degree of a graph $\mathcal{G}$. 
    Define 
    \[
        E_2 := \big\{\deg_{\max}(\Hh_{x,\rho}) \leq \sqrt{\rho}~\big\}.
    \]
    Note that $\widetilde{\mathbbm{B}}(x,\rho)$ can be embedded in a partition with $\left\lceil \frac{2e+1}{4}c_0'\rho^2\right\rceil^d$ hypercubes whose side has length $2r$. We obtain an upper bound for $\p(E_2^c)$ by considering the event in which the PPP assigns more than $\sqrt{\rho}/3^d$ points to at least one of the above-referred hypercubes. By Chernoff bound, if $X \sim \operatorname{Poi}(\lambda')$ with respect to $\mathbf{P}$, then $\mathbf{P}(X \geq s) \leq \exp\left(\lambda'(e^{s'}-1)-s's\right)$. Let $s'=\log(s/\lambda')>0$ in the previous inequality, then $\mathbf{P}(X \geq s)< (e \cdot \lambda'/s)^s$ for all $s>0$. Hence,
    \begin{equation} \label{eq:p.E2.comp.bound}
        \p\big(E_2^c\big) \leq (3c_0')^d \rho^{2d} \left( \frac{e \cdot 6^d r^d }{\sqrt{\rho}}\right)^{\sqrt{\rho}}.
    \end{equation}

    It is a well-known fact that $\deg_{\max}(\Hh)$ is $\p$-a.s. unbounded. Let $S_t^{x,\rho}$ stand for $S_t^x$ restricted to $\Hh_{x,\rho}$. Denote by $\mathsf{N}^{x}_\rho$ the number of jumps performed by $S^{x,\rho}_t$ up to time $\rho$. Let $\mathsf{N}(\rho)$ be the counting of the homogeneous PPP on $[0,\rho]$ with rate $\sqrt{\rho}$ such that $\mathsf{N}^x_\rho \le \mathsf{N}(\rho)$ on $E_2$. 
    
    Define $E_3 := \{\mathsf{N}(\rho) < \rho^2 \cdot e\}$ and let $E := E_1\cap E_2 \cap E_3$. Then, by the same inequality used above for the Poisson distributions, one has that
    \begin{equation} \label{eq:p.E3.c.bound}
        \p(E_3^c) \leq \rho^{-\rho^2/2}.
    \end{equation}
    Observe that, since $B(x, \rho^{\mathfrak{b}}/2+r\cdot e\cdot c_0' \rho^2 +r) \subseteq \widetilde{\mathbbm{B}}(x,\rho)$, it follows that $S_t^x = S_t^{x,\rho}$ on $E$ for all $t \in [0, \rho]$. We will estimate
    \begin{equation} \label{eq:St.ball.bound}
        \p\left(S_t^x \not \in B\left(x,\frac{\rho^{\mathfrak{b}}}{2}\right)\right) \le \p\left(\left\{S_t^{x,\rho} \not \in \bar{B}_D\left(x,\frac{c_0'}{2}\rho^{\mathfrak{b}}\right)\right\} \cap E_1 \cap E_2 \right) + \p(E^c).
    \end{equation}
    
    By \cref{lm:heat.kernel}, one has that, for all $u,v \in \Hh_{x, \rho}$,
    \[
        \p(S_t^{u,\rho}=v \mid E_1\cap E_2) \leq \deg_{\max}(\Hh_{x,\rho})\mathsf{q}_t\big(q(x),v\big)
    \]
    with
    %where $\mathsf{q}_t(u,v)$ is the heat kernel {\color{red}(see \citet{barlow2017} for more details). We apply Corollary 11 of \citet{davies1993} to verify that $k(q(x),v) \leq \deg_{\max}(\Hh_{x,\rho})$} and that there exist $c_3,c_4>0$ such that
    %\[
        %\mathsf{q}_t\big(q(x),v\big) \le c_3\rho^{\mathfrak{b}} \exp\left( - c_4\frac{D\big(q(x),v\big)^2}{t\sqrt{\rho}}\right)
    %\]
    \[
        \mathsf{q}_t\big(q(x),v\big) \le 2\sqrt{10}c_0'\rho^{\mathfrak{b}} \cdot\exp\left( - \frac{D\big(q(x),v\big)^2}{4t}\right)
    \]
    %on $E_1\cap E_2$.
    Let us write $\mathcalboondox{S}(x,\rho):= \left\{S_t^{x,\rho} \not \in \bar{B}_D\left(x,\frac{c_0'}{2}\rho^{\mathfrak{b}}\right)\right\} \cap E_1 \cap E_2$  and set ${\mathcalboondox{V}_{x,\rho} := \Hh_{x,\rho}\setminus \bar{B}_D(x,\frac{c_0'}{2}\rho^{\mathfrak{b}})}$. Then
    \begin{align*}
        \p\big(\mathcalboondox{S}(x,\rho) \big) &\le \E\left[\#(\mathcalboondox{V}_{x,\rho}) \cdot \sqrt{\rho} \cdot \sup_{v \in \mathcalboondox{V}_{x,\rho}}\big\{ \mathsf{q}_t(x,v)\big\} \cdot \mathbbm{1}_{(E_1\cap E_2)}\right]\\
        &\leq 2\sqrt{10}(e+1)^d \rho^{2d+{\mathfrak{b}}+1/2}\cdot\exp\left( -\frac{c_0'}{8} ~\rho^{2{\mathfrak{b}}-1} \right).
        %&\leq c_3(e+1)^d \rho^{2d+1/2}\cdot\exp\left( \frac{c_0'c_4}{2} ~\rho^{2{\mathfrak{b}}-3/2} \right).
    \end{align*}
    
    We combine the result above with \cref{prop:Hn.growth}, \eqref{eq:D.Euclid.Lm.Yao}, \eqref{eq:p.E2.comp.bound}, \eqref{eq:p.E3.c.bound}, and \eqref{eq:St.ball.bound} to verify the existence of $c_1',c_2'>0$ such that, for all $x \in \R^d$ and $\rho \geq 1$ fixed,
    \[
        \p\left(S_t^x \not \in B\left(x,\frac{\rho^{\mathfrak{b}}}{2}\right)\right) \le  \frac{c_1'}{2} \exp\left(-c_2' \rho^{2{\mathfrak{b}}-1}\right) =: \mathsf{p}_\rho.
        %\frac{c_1'}{2} \exp\left(-c_2' \rho^{2{\mathfrak{b}}-3/2}\right) =: \mathsf{p}_\rho.
    \]
    
    We now turn to the proof of \eqref{eq:first.exit.p.bound}. Let $S_\uptau:= S^x_{\uptau_\rho^x}$. Note that, since $S_\uptau$ is in the boundary of the subgraph $ \Hh \cap B(x,\rho^{\mathfrak{b}})$,
    \begin{align*}
        \p\left(\uptau^x_\rho < t\right) &= \p\big(\uptau^x_\rho < t, S^x_t \not\in B(x,\rho^{\mathfrak{b}}/2) \big) + \p\big(\uptau^x_\rho < t, S^x_t \in B(x,\rho^{\mathfrak{b}}/2)\big)\\
        &\leq \mathsf{p}_\rho + \E\left[ \mathbbm{1}_{(\uptau^x_\rho < t)} \p\big(S^{S_\uptau}_{t-\uptau^x_\rho} \in B(x,\rho^{\mathfrak{b}}/2)\big) \right]\\
        &\leq \mathsf{p}_\rho + \E\left[ \mathbbm{1}_{(\uptau^x_\rho < t)} \max_{y \in \partial B(x,\rho^{\mathfrak{b}}) \cap \Hh} \sup_{s \in [0, \rho]}\p\big(S^y_{s} \not\in B(y,\rho^{\mathfrak{b}} /2)\big) \right]\\
        &\leq \mathsf{p}_\rho + \mathsf{p}_\rho \cdot \p\left(\uptau^x_\rho < t\right).
    \end{align*}
    
    Then it follows that $\p\left(\uptau^x_\rho < t\right) \le 2\mathsf{p}_\rho$ which proves \eqref{eq:first.exit.p.bound} for $\rho\geq 1$.
    The case $\rho<1$ is covered by adjusting $c_1'$ and the proof is complete.
    \end{proof}
    
    Consider now the coalescent random walks that dominate the invasion dynamics by duality. Our aim is to determine if the blue species invades $q(x)$ up to time $t\le\rho$. Observe that any admissible path of invasion is stochastically dominated by the existence of a time-reversed random walk $S^x_{t'}$ with $t'\in[0,t]$ reaching $\upzeta$.
    
    Regard $q(x)$ as a source of random walks $\{S^{x, (i)}_{t'}: t' \in [0,t], i \in \N\}$. Set $\mathbf{n}(x)$ to be the set of neighbouring sites of $q(x)$. Write $\mathsf{N}_{x,y}^{[t]}$  for the number of points of the PPP with respect to the edge $\{q(x),y\}$ on $[0,t]$. The number of offspring of random walks starting at $q(x)$ is $\sum_{y \in \mathbf{n}(x)}\mathsf{N}^{[t]}_{x,y}$. Define the events
    \[
        E_1':=\big\{\deg\big(q(x)\big)< \sqrt{\rho}\big\}, \quad \text{and} \quad E_2':=\left\{\forall y \in \mathbf{n}\big(q(x)\big)\left(\mathsf{N}^{[t]}_{x,y}< e\cdot \rho^{3/2}\right)\right\}
    \]
    
    By applying the same upper-bounds for Poisson distributions used in the proof of \cref{cl:first.exit}, 
    \[
        \p\left((E_1')^c\right)\leq \left(\frac{ 2^d r^d }{\sqrt{\rho}}e\right)^{\sqrt{\rho}}, \quad \p\left((E_2')^c\cap E_1\right)\le \rho^{-\rho^{3/2}}.
    \]
    
    Let $E' := E_1' \cap E_2'$. then one has by \cref{cl:first.exit} that
    \[
        \p\left(\{x \in \upzeta(t)\} \cap E'\right) \le e \cdot \rho^2 \cdot c_1' \exp\left( -c_2' \rho^{2 {\mathfrak{b}} -1}\right).
    \]
    We complete the proof of \eqref{eq:invasion.q(x)} by choosing suitable $C, C'>0$.
\end{proof}

The established lemma above provides a method to control the invasion times in a region occupied by another species. The next result offers a more refined approach to studying the competition dynamics of the model. First, let us introduce some notation. Let $\bar{\upxi}, \bar{\upzeta} \subseteq \R^d$ be two disjoint sets and denote
\[P_{\bar{\upxi}, \bar{\upzeta}}(\cdot):=\p\left(~\cdot \mid\left\{\upxi(0)=\bar{\upxi}\cap\Hh, ~\upzeta(0)=\bar{\upzeta}\cap\Hh\right\}\cap\Uptheta_{w,w'}\right)\]

\begin{figure}[htb!]
    \centering
    \includegraphics[width=0.65\linewidth]{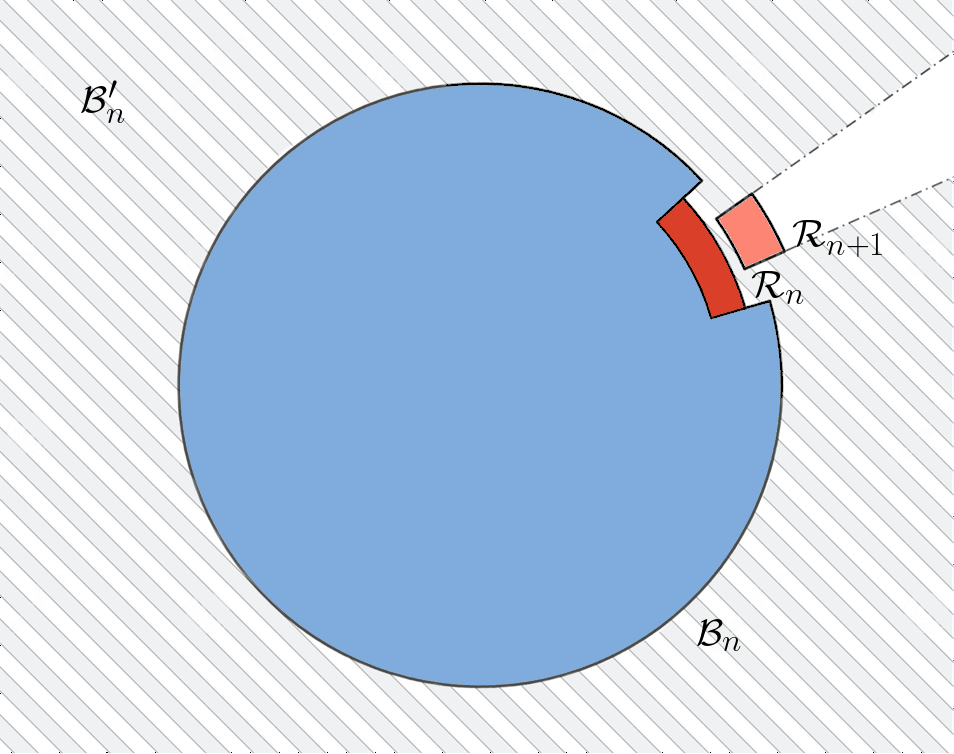}
    \caption{The regions $\mathcal{R}_n$, $\mathcal{R}_{n+1}$, $\mathcal{B}_n$, and $\mathcal{B}_n'$ before intersecting $\Hh$.}
\end{figure}
Furthermore, let us consider the following regions of $\Hh$:
\begin{align*}
\mathcal{R}_n&:=\Upphi_n(w)\\
\mathcal{B}_n &:= \left(\upvarphi \cdot B\left(o,\tfrac{1}{1+\mathfrak{d}}t_n\right) \cup\left( \left.\upvarphi \cdot B\left(o, t_n+ (t_n)^{\mathfrak{b}}\right)\right\backslash \operatorname{Cone}(w, r_n)\right)\right)\cap\Hh\text{, and}\\
    \mathcal{B}_n' &:= \left(\upvarphi \cdot B(o,t_n) \cup \big(\operatorname{Cone}(w, r_{n+1})\big)^c\right) \cap \Hh.
\end{align*}

Now, we can obtain the following property regarding the competition model.

\begin{lemma} \label{lm:control.species.growth}
    There exist constants $C,C'>0$  such that the following holds for all  $t_0>\bar{s}$. If ~$~\mathcal{R}_n \subseteq \bar{\upxi}$ and $\bar{\upzeta} \subseteq \mathcal{B}_n$, then
    \begin{equation} \label{eq:upper.bdn.blue.steps}
        P_{\bar{\upxi},\bar{\upzeta}}\big(\upzeta(\mathfrak{d}t_n)\not\subseteq \mathcal{B}_n'\big) \leq C t_n^{3d}\exp\big(-C'(t_n)^{2\mathfrak{b}-3/2}\big).
    \end{equation}
\end{lemma}

\begin{proof}
    The demonstration closely follows the proof of Lemma 3 in \cite{kordzakhia2005}, with an emphasis on the key differences and necessary adjustments for our model with random structures.
    
    The sets $\mathcal{R}_n$, $\mathcal{R}_{n+1}$, $\mathcal{B}_n$, $\mathcal{B}_n'$ correspond to $\mathscr{R}_0^n$, $\mathscr{R}_1^n$, $\mathscr{B}_0^n$, $\mathscr{B}_1^n$ from \cite{kordzakhia2005}, respectively, scaled by $\upvarphi$ and intersected with $\Hh$. Similarly, the constants $\mathfrak{a}$, $\mathfrak{b}$, and $\mathfrak{d}$ represent $\alpha$, $\beta$, and $\delta$ in the same lemma.

    Their proof consists of five claims. Since the limiting shape is a Euclidean ball, which is uniformly curved, Claim 1 is immediately satisfied. The remaining four claims depend on the polynomial growth of $\Hh$, estimated by \cref{prop:Hn.growth}, and the moderate deviation given by \cref{thm_new_moderate_deviations}. Noting that $0 < 2\mathfrak{b} - \frac{3}{2} < \mathfrak{b} - \frac{1}{2}$ ensures that the upper bounds for the probabilities in Claims 2 and 3 hold true. In Claim 4, we replace the application of Lemma 1 in \cite{kordzakhia2005} with \cref{lm:control.invasion}, resulting in \eqref{eq:upper.bdn.blue.steps}.
\end{proof}

\section{The coexistence of the species} \label{sec:proof}

In this section, we address the theorem concerning the coexistence of species in the competition model. This result is derived from the intermediate condition, combined with the lemmas established in the previous section, moderate deviations, and the fact that the limiting shape is a Euclidean ball.

\begin{proof}[Proof of \cref{thm:coexistence}]
    First, since the sets $W$ and $W'$ determine a viable initial configuration, \cref{lm:intermediate.condition} ensures that the intermediate condition \eqref{intermediate.condition} holds true. Let $P(\cdot):=\p(~\cdot \mid\Uppsi_{w,w'}\big)$. We begin by finding a lower bound for the probability $\p\big(\mathrm{Coex}(\upxi,\upzeta)\big) \geq P\big(\mathrm{Coex}(\upxi,\upzeta)\big) \cdot \p(\Uppsi_{w,w'})$. Suppose, w.l.o.g., that \[P(\text{red dies out})\geq P(\text{blue dies out}),\] where `\emph{red dies out}' is the event when there exists a time $t>0$ with $\upxi(t)=\varnothing$ (similar for `\emph{blue dies out}'). Note that Boole's inequality yields
    \begin{align}
        P\big(\mathrm{Coex}(\upxi,\upzeta)\big) & \geq 1 - P\left( \exists t \geq 0 \ \text{s.t.} \ \upxi(t) = \varnothing \ \text{or} \ \upzeta(t) = \varnothing \right) \nonumber\\
        & \geq 1 - 2 P\left( \exists t \geq 0 \ \text{s.t.} \ \upxi(t) = \varnothing \right). \label{eq:coex.bound}
    \end{align}
    Let us prove that, for a $\varepsilon \in (0,1/4)$,
    \begin{equation*}
        P(\text{red dies out})=P\left( \exists t \geq 0 \ \text{s.t.} \ \upxi(t) = \varnothing \right) < 2 \varepsilon
    \end{equation*}
    to establish the theorem. Let $\Gamma:=\bigcap_{n\in\N_0} \Gamma_n$ and repeat the arguments in the proof of \cref{lm:intermediate.condition} to select a large $s_0>\bar{s}$ such that, for all $t_0\geq s_0$,
    \begin{equation} \label{eq:first.small.compr.bdn}
        P(\Gamma^c) \leq \sum_{n \in \N} P\big(\Gamma_n^c\big) < \varepsilon.
    \end{equation}
    Next, by applying \cref{lm:control.species.growth}, we find a possibly larger $s_0$ so that, for all $t_0 \geq s_0$,
    \begin{equation} \label{eq:second.small.compr.bdn}
    \sum_{n \in \N_0}P_{\mathcal{R}_n,\mathcal{B}_n}\big(\upzeta(\mathfrak{d}t_n)\not\subseteq \mathcal{B}_n'\big) <\varepsilon.
    \end{equation}
   
    By stochastic domination, \eqref{eq:first.small.compr.bdn} and \eqref{eq:second.small.compr.bdn} imply
    \begin{equation*}
        P\left( \text{red dies out} \right) \leq P(\Gamma^c)+\sum_{n\in \N_0}  P\big(\{\upzeta(t_{n+1})\not\subseteq \mathcal{B}_n'\}\cap\Gamma \mid \upzeta(t_{n})\subseteq \mathcal{B}_n\big)< 2 \varepsilon.
    \end{equation*}
    
   The inequality above leads to  $\p\big(\mathrm{Coex}(\upxi,\upzeta)\big) >(1-4\varepsilon)\p(\Uppsi_{w,w'})>0$, which is the statement of the theorem.
\end{proof}

\section*{Acknowledgements}
This research was supported by grants \#2017/10555-0, \#2019/19056-2, \#2020/12868-9, \#2023/13453-5, and \#2024/06021-4, S\~ao Paulo Research Foundation (FAPESP).  Additionally, L.R. de Lima would like to express gratitude to the Department of Mathematics at the Bernoulli Institute, University of Groningen, and the Department of Statistics at the University of Warwick, where the foundations of \cite{delima2024} and this article were developed during his visit. Their warm hospitality is sincerely appreciated. Special thanks are also due to Daniel Valesin for his invaluable contributions to the conception of this problem and for inspiring discussions that shaped this work.

%%%%%%%%%%%%--BIBLIOGRAPHY--%%%%%%%%%%%%%%%
\bibliographystyle{abbrvnat}
\bibliography{references}

\begin{thebibliography}{15}
\providecommand{\natexlab}[1]{#1}
\providecommand{\url}[1]{\texttt{#1}}
\expandafter\ifx\csname urlstyle\endcsname\relax
  \providecommand{\doi}[1]{doi: #1}\else
  \providecommand{\doi}{doi: \begingroup \urlstyle{rm}\Url}\fi

\bibitem[Barlow(2017)]{barlow2017}
M.~T. Barlow.
\newblock \emph{Random walks and heat kernels on graphs}, volume 438 of \emph{London Mathematical Society Lecture Note Series}.
\newblock Cambridge University Press, Cambridge, 2017.
\newblock ISBN 978-1-107-67442-4.
\newblock \doi{10.1017/9781107415690}.
\newblock URL \url{https://doi.org/10.1017/9781107415690}.

\bibitem[Coletti et~al.(2023)Coletti, de~Lima, Hinsen, Jahnel, and Valesin]{coletti2023}
C.~F. Coletti, L.~R. de~Lima, A.~Hinsen, B.~Jahnel, and D.~Valesin.
\newblock Limiting shape for first-passage percolation models on random geometric graphs.
\newblock \emph{Journal of Applied Probability}, 60\penalty0 (4):\penalty0 1–19, 2023.
\newblock \doi{10.1017/jpr.2023.5}.
\newblock URL \url{https://doi.org/10.1017/jpr.2023.5}.

\bibitem[Davies(1993)]{davies1993}
E.~B. Davies.
\newblock Large deviations for heat kernels on graphs.
\newblock \emph{J. London Math. Soc. (2)}, 47\penalty0 (1):\penalty0 65--72, 1993.
\newblock ISSN 0024-6107.
\newblock \doi{10.1112/jlms/s2-47.1.65}.
\newblock URL \url{https://doi.org/10.1112/jlms/s2-47.1.65}.

\bibitem[de{ }Lima(2024)]{delima2024}
L.~R. de{ }Lima.
\newblock \emph{Asymptotic Shape of Subadditive Processes on Groups and on Random Geometric Graphs}.
\newblock PhD thesis, Universidade Federal do ABC, 2024.

\bibitem[de{ }Lima and Valesin(2024)]{delima2024speed}
L.~R. de{ }Lima and D.~Valesin.
\newblock Speed of convergence and moderate deviations of {FPP} on random geometric graphs.
\newblock \emph{arXiv preprint arXiv:2411.02046}, 2024.
\newblock URL \url{https://doi.org/10.48550/arXiv.2411.02046}.

\bibitem[Durrett and Neuhauser(1997)]{durrett1997}
R.~Durrett and C.~Neuhauser.
\newblock Coexistence results for some competition models.
\newblock \emph{Ann. Appl. Probab.}, 7\penalty0 (1):\penalty0 10--45, 1997.
\newblock ISSN 1050-5164.
\newblock \doi{10.1214/aoap/1034625251}.
\newblock URL \url{https://doi.org/10.1214/aoap/1034625251}.

\bibitem[Durrett et~al.(2020)Durrett, Junge, and Tang]{durret2020}
R.~Durrett, M.~Junge, and S.~Tang.
\newblock Coexistence in chase-escape.
\newblock \emph{Electron. Commun. Probab.}, 25:\penalty0 Paper No. 22, 14, 2020.
\newblock \doi{10.1214/20-ecp302}.
\newblock URL \url{https://doi.org/10.1214/20-ecp302}.

\bibitem[Garet and Marchand(2005)]{garet2005}
O.~Garet and R.~Marchand.
\newblock Coexistence in two-type first-passage percolation models.
\newblock \emph{Ann. Appl. Probab.}, 15\penalty0 (1A):\penalty0 298--330, 2005.
\newblock ISSN 1050-5164.
\newblock \doi{10.1214/105051604000000503}.
\newblock URL \url{https://doi.org/10.1214/105051604000000503}.

\bibitem[H\"{a}ggstr\"{o}m and Pemantle(1998)]{haggstrom1998}
O.~H\"{a}ggstr\"{o}m and R.~Pemantle.
\newblock First passage percolation and a model for competing spatial growth.
\newblock \emph{J. Appl. Probab.}, 35\penalty0 (3):\penalty0 683--692, 1998.
\newblock ISSN 0021-9002.
\newblock \doi{10.1239/jap/1032265216}.
\newblock URL \url{https://doi.org/10.1239/jap/1032265216}.

\bibitem[Kordzakhia and Lalley(2005)]{kordzakhia2005}
G.~Kordzakhia and S.~P. Lalley.
\newblock A two-species competition model on {$\Bbb Z^d$}.
\newblock \emph{Stochastic Process. Appl.}, 115\penalty0 (5):\penalty0 781--796, 2005.
\newblock ISSN 0304-4149.
\newblock \doi{10.1016/j.spa.2004.12.003}.
\newblock URL \url{https://doi.org/10.1016/j.spa.2004.12.003}.

\bibitem[Lanchier(2024)]{lanchier2024}
N.~Lanchier.
\newblock \emph{Stochastic Interacting Systems in Life and Social Sciences}.
\newblock De Gruyter, Berlin, Boston, 2024.
\newblock ISBN 9783110791884.
\newblock \doi{doi:10.1515/9783110791884}.
\newblock URL \url{https://doi.org/10.1515/9783110791884}.

\bibitem[Neuhauser(1992)]{neuhauser1992}
C.~Neuhauser.
\newblock Ergodic theorems for the multitype contact process.
\newblock \emph{Probab. Theory Related Fields}, 91\penalty0 (3-4):\penalty0 467--506, 1992.
\newblock ISSN 0178-8051.
\newblock \doi{10.1007/BF01192067}.
\newblock URL \url{https://doi.org/10.1007/BF01192067}.

\bibitem[Penrose and Pisztora(1996)]{penrose1996}
M.~Penrose and A.~Pisztora.
\newblock Large deviations for discrete and continuous percolation.
\newblock \emph{Adv. in Appl. Probab.}, 28\penalty0 (1):\penalty0 29--52, 1996.
\newblock ISSN 0001-8678.
\newblock \doi{10.2307/1427912}.
\newblock URL \url{https://doi.org/10.2307/1427912}.

\bibitem[Tang et~al.(2018)Tang, Kordzakhia, and Lalley]{tang2018phase}
S.~Tang, G.~Kordzakhia, and S.~P. Lalley.
\newblock Phase transition for the chase-escape model on 2d lattices.
\newblock \emph{arXiv preprint arXiv:1807.08387}, 2018.

\bibitem[Yao et~al.(2011)Yao, Chen, and Guo]{yao2011}
C.-L. Yao, G.~Chen, and T.-D. Guo.
\newblock Large deviations for the graph distance in supercritical continuum percolation.
\newblock \emph{J. Appl. Probab.}, 48\penalty0 (1):\penalty0 154--172, 2011.
\newblock ISSN 0021-9002.
\newblock \doi{10.1239/jap/1300198142}.
\newblock URL \url{https://doi.org/10.1239/jap/1300198142}.

\end{thebibliography}

\end{document}